\documentclass[a4paper,10pt]{amsart}

\usepackage[utf8]{inputenc}
\usepackage[usenames, dvipsnames]{color}
\newtheorem{theorem}{Theorem}
\newtheorem{lemma}{Lemma}

\newcommand{\C}{\mathbb C}
\newcommand{\R}{\mathbb R}
\newcommand{\N}{\mathbb N}
\newcommand{\Z}{\mathbb Z}

\title[Approximating the Riemann Zeta-function by Polynomials]{Approximating the Riemann Zeta-function by Polynomials with Restricted Zeros}

\author{P. M. Gauthier}

\address{Département de mathématiques et de statistique, Université de Montréal,
CP-6128 Centreville, Montréal,  H3C3J7, CANADA}
\email{gauthier@dms.umontreal.ca}

\begin{document}

\begin{abstract}
We approximate the Riemann Zeta-Function by polynomials  and  Dirichlet polynomials with restricted zeros. 
\end{abstract}

\keywords{Riemann zeta-function, Riemann Hypothesis} \subjclass{Primary: 30E15; Secondary: 11M26}

\thanks{Research supported by NSERC (Canada) grant RGPIN-2016-04107}

\maketitle

The Riemann zeta-function has zeros at the negative even integers (the so-called trivial zeros). The Riemann Hypothesis states that the remaining zeros (the non-trivial zeros) all lie on the critical line $S= \{z: \Re z = 1/2\}.$ A refinement of the Riemann Hypothesis claims that moreover the zeros are simple. 

We wish to approximate $\zeta$ by sequences of polynomials whose zeros have these properties on larger and larger sets.  Since Euler originally defined the zeta-function by the Dirichlet series
$$
	\zeta(x) = \sum_{n=1}^\infty\frac{1}{n^x},  \quad 1<x<+\infty,
$$
it seems natural to approximate $\zeta,$ not only by ``ordinary" polynomials
$$
	P(z) = \sum_{n=0}^ma_nz^n,
$$
but also by {\em Dirichlet} polynomials
$$
	D(z) = \sum_{n=1}^m\frac{a_n}{n^z}.
$$
To clearly distinguish between Dirichlet polynomials and ordinary polynomials, 
we shall sometimes refer to the latter as  {\em algebraic} polynomials. 

While the theory of approximation by algebraic polynomials is a well developed classical subject, that of approximation by Dirichlet polynomials has received less attention. Recently \cite[Lemma 4.1]{A}, it has been shown that the two theories are in fact equivalent.

\begin{theorem}
There exists an increasing sequence $K_n$ of compact subsets of $\C$ whose union is $\C,$ a sequence $P_n$ of algebraic polynomials and a sequence $D_n$ of Dirichlet polynomials, with the following properties.
\begin{equation}
	\max\big\{|P_n(z)-\zeta(z)|,|D_n(z)-\zeta(z)|\big\}\le 1/n, \quad \mbox{for} 
		\quad z\in K_n\setminus\{1\}.
\end{equation}
\begin{equation}
	P_n(1)=D_n(1)| = n.
\end{equation}
On $K_n\cap(\R\cup S),$
\begin{equation}
	 P_n \quad \mbox{and} \quad D_n \quad \mbox{have only simple zeros and they are the} \quad \zeta-zeros.
\end{equation}
On $K_n\setminus (\R \cup S),$
\begin{equation} 
	P_n \quad \mbox{and} \quad D_n \quad \mbox{have no zeros.}
\end{equation}
\begin{equation}
	P_n(\R)\subset \R, \quad D_n(\R)\subset \R. 
\end{equation}

\end{theorem}

It follows that $P_n\rightarrow \zeta$ and $D_n\rightarrow \zeta$ pointwise on all of $\C$ and, for each fixed $m,$ the convergence is  uniform on $K_m\setminus\{1\}$ and
spherically uniform on $K_m.$ 
 
Let $K$ be a compact subset od $\C,$ As usual,  $A(K)$ denotes the space of functions continuous on $K$ and holomorphic on the interior $K^0,$ endowed with the sup-norm, and $P(K)$ denotes the uniform closure in $A(K)$ of the set of algebraic polynomials. Similarly, $D(K)$ will denote the closrure in $A(K)$ of the set of Dirichlet polynomials. 

\begin{lemma}\label{Mergelyan}
For a compact set $K\subset\C$ the following are equivalent:

a) $\C\setminus K$ is connected;

b) $P(K)=A(K);$

c) $D(K)=A(K).$
\end{lemma}

The equivalence of a) and b) is Mergelyan's Theorem, the most important theorem in polynomial approximation. The equivalence of b) and c) is a very recent result \cite[Lemma 4.1]{A}, due to Aron et al. 

For $A\subset \C,$ we denote by $A^*$ the set $\{\overline z: z\in A\}$ and we say that $A$ is real-symmetric if $A=A^*.$ We say that a function $f:A\rightarrow \C$ on a real-symmetric set $A$ is real-symmetric if $f(z)=\overline{f(\overline z)}.$ For a class $X$ of functions on a real-symmetric set $A,$ we denote by $X_\R$ the class of functions in $X$ which are real-symmetric. If $X$ is a complex vector space, we note that $X_\R$ is a real vector space (even though the functions may  be complex valued).   We have a real-symmetric version of the previous lemma. 

\begin{lemma}\label{symmetric Mergelyan}
For a real-symmetric compact set $K\subset\C,$ the following are equivalent:

a) $\C\setminus K$ is connected;

b) $P_\R(K)=A_\R(K);$

c) $D_\R(K)=A_\R(K).$
\end{lemma}

\begin{proof}
Suppose $\C\setminus K$ is not connected. Then $K$ has a bounded complementary component $U.$ Fix $a\in U.$ The function $f(z) = (z-a)^{-1}(z-\overline a)^{-1}$ is in $A_\R(K).$  
Suppose, to obtain a contradiction, that there is a sequence $p_n$ of real-symmetric polynomials such that 
$$
	|p_n(z)- (z-a)^{-1}(z-\overline a)^{-1}|<1/n, \quad \forall z\in K. 
$$
Then 
$$
	|p_n(z)(z-a)(z-\overline a)-1|<|(z-a)(z-\overline a)|/n, 
		\quad \forall z\in \partial (U\cup U^*)\subset K . 
$$
By the mazimum principle, 
$$
|p_n(z)(z-a)(z-\overline a)-1|<\max_{w\in \partial (U\cup U^*)}|(w-a)(w-\overline a)|/n, 
		\quad \forall z\in U\cup U^*. 
$$
In particular, at $z=a,$ we have
$$
	1 < \max_{w\in \partial (U\cup U^*)}|(w-a)(w-\overline a)|/n, \quad \forall n,
$$
which is a absurd. Therefore b) implies a).  A similar argument shows that c) implies a). 

Now suppose that $\C\setminus K$ is connected and $f\in A_\R(K).$ By Lemma \ref{Mergelyan}, there are algebraic polynomials $P_n$ and Dirichlet polynomials $D_n$ which converge uniformly to $f$ on $K.$ Since $f$ is real-symmetric, it is easy to see that the real-symmetric algebraic polynomials $(P_n(z)+\overline{P_n(\overline z})/2$ and the real-symmetric Dirichlet polynomials $(D_n(z)+\overline{D_n(\overline z})/2$ also converge uniformly to $f$ on $K.$ Thus, a) implies b) and c). 
\end{proof}

The next lemma, due to Frank Deutsch \cite{D},  generalizes a result of Walsh and  states that if we can approximate we can simultaneously interpolate. 

\begin{lemma}\label{Walsh}
Let $Y$ be a dense (real or complex) linear subspace of the (respectively real or complex) linear topological space $X$ and let $L_1,\ldots,L_n$ be continuous linear functionals on $X.$ Then for each $x\in X$ and each neighbourhood $U$ of $x$ there is a $y\in Y$ such that $y\in U$ and $L_i(y)=L_i(x), i=1,\ldots,n.$
\end{lemma}

With the help of these  lemmas, we now prove the theorem. 

\begin{proof}
Our construction of the sets $K_n$ is inspired by a construction in \cite{G}.

First, we prove the theorem for algebraic polynomials $P_n.$ 
Set $t_0=0$ and let $t_k, k\in \N,$ be the imaginary parts of the zeros of $\zeta$ in the upper half-plane, arranged in increasing order. If the Riemann Hypothesis fails, there may be $t_k, k>0,$ corresponding to more than one zero of $\zeta.$ Choose $0<\lambda_1<\lambda_2<\cdots<1,$ such that $\lambda_j\nearrow 1.$ Let $s_0=1$ and for $k\in\Z\setminus\{1\},$ let $s_k=2k.$ For each $i\in\N$ and for each  $-i\le j \le i,$ and $k=0,1,\ldots,i,$ set
$$
	Q_{ijk} = \left\{z:s_j\le \Re z \le \lambda_i(s_{j+1}-s_j), \,  t_k\le |\Im z|\le \lambda_i(t_{k+1}-t_k) \right\}
$$
and 
$$
	Q_i = \bigcup Q_{ijk}, \quad  -i\le j\le i,\quad  k=0,1,\ldots, i.
$$
The  compact set $Q_i$ is real-symmetric and is  the union of  disjoint closed rectangles,  so $\C\setminus Q_i$ is connected.  

Let $Z_i^1$ be the zeros of $\zeta$ in $Q_i\cap(\R\cup S)$ and let $Z_i^2$ be the zeros of $\zeta$ in $Q_i\setminus(\R\cup S).$ Then $Z_i=Z_i^1\cup Z_i^2$ is the set of zeros of $\zeta$ in $Q_i.$ 
Denoting by $B(z,r)$ (respectively  $\overline B(z,r)$) the open (respectively closed) disc of center $z$ and radius $r,$ set
$$
	\mathcal B_i = \bigcup_{z\in Z_i\cup\{1\}} B\left(z,\frac{1}{i}\right)
$$
and
$$
	K_i = (Q_i \setminus \mathcal B_i)\cup Z_i \cup\{1\}.
$$
Then $K_1, K_2, \ldots,$ is an increasing sequence of compact sets whose union is $\C$ and the complement of each $K_n$ is connected.  Moreove, since $Z_i^1$ and $Z_i^2$ are real-symmetric, so are the $K_i.$ Now, for $n=1,2,\ldots,$ set
$$
\mathcal K_n = K_n \cup \bigcup_{z\in Z_n\cup\{1\}}\overline B\left(z,\frac{1}{2n}\right) = 
(Q_n\setminus \mathcal B_n)\cup \bigcup_{z\in Z_n\cup\{1\}}\overline B\left(z,\frac{1}{2n}\right).
$$

For each $n,$ the complement of $\mathcal K_n$ is connected and $K_n$ is real-symmetric and so, by Lemma \ref{symmetric Mergelyan}, the real-symmetric algebraic  polynomials are dense in the space of  real-symmetric holomorphic functions on (neighbourhoods of) $\mathcal K_n.$ By Lemma \ref{Walsh}, for every real-symmetric function $f$ holomorphic on $\mathcal K_n,$ and finitely many points $a_1,\ldots,a_m\in\mathcal K_n$ and for each $\epsilon>0,$ there is a real-symmetric polynomial $P,$ such that $|f-P|<\epsilon$ on $\mathcal K_n,$ and $P(a_j)=f(a_j), \, j=1,\ldots, m.$ Moreover, for each  $k\in \N,$ there is such a polynomial $P$ such that, for each $a_j\in \mathcal K_n^0,$   
$P^{(\ell)}(a_j) = f^{(\ell)}(a_j),$ for $\ell =0,1,\ldots, k.$ 

We shall apply this approximation and interpolation procedure to the following function, holomorphic and real-symmetric on $\mathcal K_n:$
$$
	f_n(z) = 
\left\{
	\begin{array}{cll}
\zeta(z)	&	\mbox{for} \quad z\in Q_n\setminus \mathcal B_n, & \\
n  		&	\mbox{for} \quad z\in  \overline   B\big(1,1/(2n)\big),\\
z-a		&	\mbox{for} \quad z\in \overline  B\big(a,1/(2n)\big), & a\in Z^1_n,\\
1/n 	&  \mbox{for} \quad z\in \overline   B\big(a,1/(2n)\big), & a\in Z^2_n.
	\end{array}
\right. 
$$
Set $\delta_n=\min |f_n(z)|,$ for $z\in Q_n\setminus\mathcal B_n.$ Since $\zeta$ has no zeros on this compact set, $\delta_n>0.$  Choose $\epsilon_n<\min\{\delta_n/2,1/n\}.$ 
Invoking the approximation-interpolation procedure, for each $n,$ there is a real-symmetric polynomial $P_n,$ such that  
$$
	|P_n(z)-f_n(z)|<\epsilon_n, \quad \mbox{for all} \quad  z\in \mathcal K_n;
$$ 
$$
	P_n(1)=n;
$$
$$
	P_n(a)=0, \, P_n^\prime(a)=1, \quad \mbox{for all} \quad a\in Z^1_n;
$$
$$
	P_n(a) = 1/n,  \quad \mbox{for all} \quad a\in Z^2_n.
$$
Since $P_n$ is real-symmetric, $P_n(\R)\subset \R.$ 
This completes the proof for algebraic polynomials. 

The proof for Dirichlet polynomials is identical (thanks to Lemmas \ref{symmetric Mergelyan} and \ref{Walsh}).

\end{proof}



\begin{thebibliography}{11}

\bibitem{A}
Aron, Richard M.; Bayart, Frédéric; Gauthier, Paul M.; Maestre, Manuel; Nestoridis, Vassili.
Dirichlet approximation and universal Dirichlet series. 
Proc. Amer. Math. Soc. {\bf 145} (2017), no. 10, 4449-4464. 

\bibitem{D} Deutsch, Frank Simultaneous interpolation and approximation in topological linear spaces. SIAM J. Appl. Math. {\bf 14} 1966 1180-1190.

\bibitem{G}  Gauthier, Paul M. Approximation of and by the Riemann zeta-function. Comput. Methods Funct. Theory {\bf 10} (2010), no. 2, 603-638.

\end{thebibliography}
\end{document}